\documentclass[10pt]{amsart}
\usepackage{amsmath,amssymb,amsthm}

\usepackage{float}
\usepackage[all,cmtip]{xy}
\usepackage{tikz}
\usetikzlibrary{matrix}
\usetikzlibrary{cd}
\usepackage{subfig}
\usepackage[colorlinks,hyperindex,linkcolor=blue]{hyperref}
\hypersetup{
 pdfauthor = {Javier Martínez-Aguinaga},
 pdftitle= {Fat distributions with Reeb directions need not be complex contact},
 pdfsubject = {Differential Geometry, Geometric Topology, Distributions on Manifolds}}

\newcommand{\R}{{\mathbb{R}}}

\newcommand{\SD}{{\mathcal{D}}}

\newcommand{\Hom}{{\text{Hom}}}
\newcommand{\Sym}{{\text{Sym}}}

\newcommand{\SL}{{\mathcal{L}}}

\newcommand{\Op}{{\mathcal{O}p}}

\newcommand{\Id}{{\operatorname{Id}}}

\newcommand{\std}{{\operatorname{std}}}

\newcommand{\Diff}{{\operatorname{Diff}}}

\newtheorem{proposition}{Proposition}[section]
\newtheorem{theorem}[proposition]{Theorem}
\newtheorem{definition}[proposition]{Definition}
\newtheorem{lemma}[proposition]{Lemma}

\newtheorem{remark}[proposition]{Remark}
\newtheorem{question}[proposition]{Question}
\newtheorem{example}[proposition]{Example}

\makeatletter
\newcommand{\superimpose}[2]{%
  {\ooalign{$#1\@firstoftwo#2$\cr\hfil$#1\@secondoftwo#2$\hfil\cr}}}
\makeatother

\title{Fat distributions with Reeb directions need not be complex contact}


\author{Javier Mart\'{i}nez-Aguinaga}
\address{Universidad Complutense de Madrid, Departamento de Álgebra, Geometría y Topología, Facultad de Ciencias
Matemáticas. 28040 Madrid, Spain}
\email{frmart02@ucm.es}

\emergencystretch=\maxdimen
\begin{document}

\begin{abstract} 

It is well known that every complex contact $3$-manifold, when regarded as a real manifold, gives rise to a fat $(4,6)$-distribution that admits two Reeb directions. Nonetheless, it was an open question whether the converse was true. This was not known even at the level of germs. The present work completely answers this question in the negative. We construct the first example of a fat distribution with two Reeb directions that does not support a complex contact structure anywhere, not even locally nor up to diffeomorphism.  This result answers an open question by Aritra Bhowmick.

\medskip

\end{abstract}
\maketitle
\addtocontents{toc}{\setcounter{tocdepth}{1}}

\section{Introduction}

Fat distributions in dimension $6$ 
 constitute a relevant class of geometric structures with much recent activity for different reasons. On one hand, they are closely related to complex contact $3$-manifolds, representing the real counterpart of that geometry \cite{CE2003}. On the other hand, they possess outstanding local properties that allow to establish local $h$-principle results for their horizontal submanifolds \cite{Bhow2024,
BhowThesis, BhDa}.

Every complex contact $3$-manifold defines a fat $(4,6)$ real distribution equipped with two transverse directions called \textit{Reeb directions}. It was an open question whether the converse holds; i.e. whether every real fat $(4,6)$-distribution with two Reeb directions is induced by a complex contact structure. This question, which we answer in this work, was open even at the local level; i.e. no counterexample was known even for germs of distributions (see Question \ref{question}).

The work \cite{Bhow2024} of A. Bhowmick showed that, in addition, fat $(4,6)$-distributions with Reeb directions also represent an outstanding subclass of distributions with remarkable local properties, proving a local $h-$principle for horizontal maps, as well as the existence of germs of horizontal surfaces into such distributions.

It is worth noting that there are actually limited flexibility results at the level of horizontal submanifolds for other types of distributions beyond the contact-theoretical framework (see \cite{dPT, MdP2021, dPS} and \cite[Ch. 22]{EM}). Alternatively, the work \cite{BhDa} by A. Bhowmick and M. Datta yields substantial flexibility results for horizontal embeddings in fat $(4,6)$-distributions by different methods but they do not cover the local $h$-principle for horizontal surfaces  \cite{Bhow2024}, whose proof strongly relies on the presence of Reeb directions. All this together highlights the importance of understanding local and global properties of fat $(4,6)$-distributions with Reeb directions as a distinguished class of distributions.

It is a fundamental question whether fat $(4,6)$-distributions with Reeb directions are locally the same as complex contact structures or, on the contrary, whether they form a strictly larger class. This motivated the following question by A. Bhowmick (which appears in the form of Question 6.2.2 in \cite{BhowThesis} and in the form of Question 2.11 in \cite{Bhow2024}).

 \begin{question}[A. Bhowmick, \cite{BhowThesis, Bhow2024}]\label{question} Is every (germ of) corank 2 fat distribution on $\R^6$, which admits local Reeb
directions, diffeomorphic to the germ of the distribution underlying a holomorphic contact
structure?\end{question}

This work provides a complete answer to Question \ref{question}
in the negative. We construct the first example of a global fat $(4,6)$-distribution on $\R^6$ with two Reeb directions that does not support a complex contact structure, not even locally (i.e. on any open set around any point) nor up to diffeomorphism (see Theorem \ref{thm:global}).

We conclude that fat distributions with Reeb directions are not generally modelled on the complex contact case and they thus form a strictly broader subclass of distributions. Our result implies that this already occurs at the level of germs. This is consistent with A. Bhowmick's expectations \cite{Bhow2024}, as he comments  that the $1$-forms obtained in \cite{Ge} indicate the presence of function moduli. This \textit{existence} result can be phrased as Theorem \ref{thm: existence}. A detailed and expanded version is presented as Theorem \ref{thm:global} in Section \ref{sec:3}. 

\begin{theorem}\label{thm: existence}
    There exists a global fat $(4,6)$-distribution $(\R^6,\SD)$ with Reeb directions that does not support a complex contact structure anywhere, not even locally nor up to diffeomorphism.
\end{theorem}

The answer to Question \ref{question} also yields consequences from the perspective of complex geometry. It is known that the complex structure on a $3$-dimensional complex contact manifold is solely determined by the underlying contact distribution \cite{CE2003}. Therefore, in line with the discussion above, answering in the negative Question \ref{question} shows that real fat $(4,6)$-distributions with Reeb directions need not determine an underlying complex structure on the manifold. Additionally, the answer to this question also yields implications at the level of horizontal immersions.

A. Bhowmick comments in \cite[p. 133]{Bhow2024}:

\begin{quote} ``Note that if the answer to the above question is in the affirmative, we can characterize germs of horizontal immersions, given by the $1$-jet prolongation of holomorphic maps $\mathbb{C}\to\mathbb{C}$''. 
\end{quote}

Thus, since we answer Question \ref{question} in the negative, this proposed characterization of germs of horizontal immersions as $1$-jet prolongation of holomorphic maps does not follow.

Our construction is subdivided into a number of steps. We first construct a semi-global structure (Theorem \ref{mainthm}) in Section \ref{sec:3}; i.e. a distribution on $\R^6$ that is fat only in a proper open subset of $\R^6$ and which possesses two Reeb directions. The proof that it does not support a complex contact structure relies on a result by A. Čap and M. Eastwood \cite{CE2003} that can be found in Section \ref{Sec:1}. Their work shows that every fat $(4,6)$-distribution on an oriented $6$-dimensional manifold $M$ canonically defines an almost complex structure $J:TM\to TM$. Moreover, the obstruction for this structure to being integrable and induced by a complex contact structure is measured by a tensor $S$. We will carefully examine this tensor for the particular case of our construction and we will show that it precludes the existence of an underlying complex structure. We will later globalise this construction in order to obtain a globally defined fat distribution with Reeb directions not supporting a complex contact structure anywhere (Theorem \ref{thm:global}).

\textbf{Acknowledgements:} The author would like to thank Álvaro del Pino and Ángel González-Prieto for useful comments. The author acknowledges support from PID2022-142024NB-I00 by MICINN (Spain).

\subsection*{Use of AI}
 The author used ChatGPT for proofreading,  computations (among other things, e.g. for finding the decomposition from Lemma \ref{lemma:dpsi}) and routine checks, which have all then been revised and verified by the author. 

\section{Further open questions}

The work \cite{Mont} of R. Montgomery implies that a generic $(4,6)$-distribution germ cannot admit a local framing generating a finite dimensional Lie algebra. This contrasts with the case of fat distribution germs \cite{Bhow2024}, which all have the complex Heisenberg Lie algebra as their nilpotentisation (see the work \cite{CFS} of S. Console, A. Fino and E. Samiou). 

Furthermore, as noted by A. Bhowmick in \cite{Bhow2024}, there exist fat distributions which are non-diffeomorphic to the complex contact one, although it was not known whether they admitted Reeb directions. Now that Theorem \ref{thm: existence} clarifies this issue and shows that there exist representatives of such distribution germs admitting Reeb directions, we can raise some related questions.

Possessing Reeb directions is not a homotopy-invariant property but it is a $\Diff$-invariant property. It would thus be interesting to understand further analytic and topological properties of the inclusion of the space of fat distribution-germs with Reeb directions within the class of all fat distribution-germs. 
\begin{figure}[h!]
	\centering
	\includegraphics[width=0.6\textwidth]{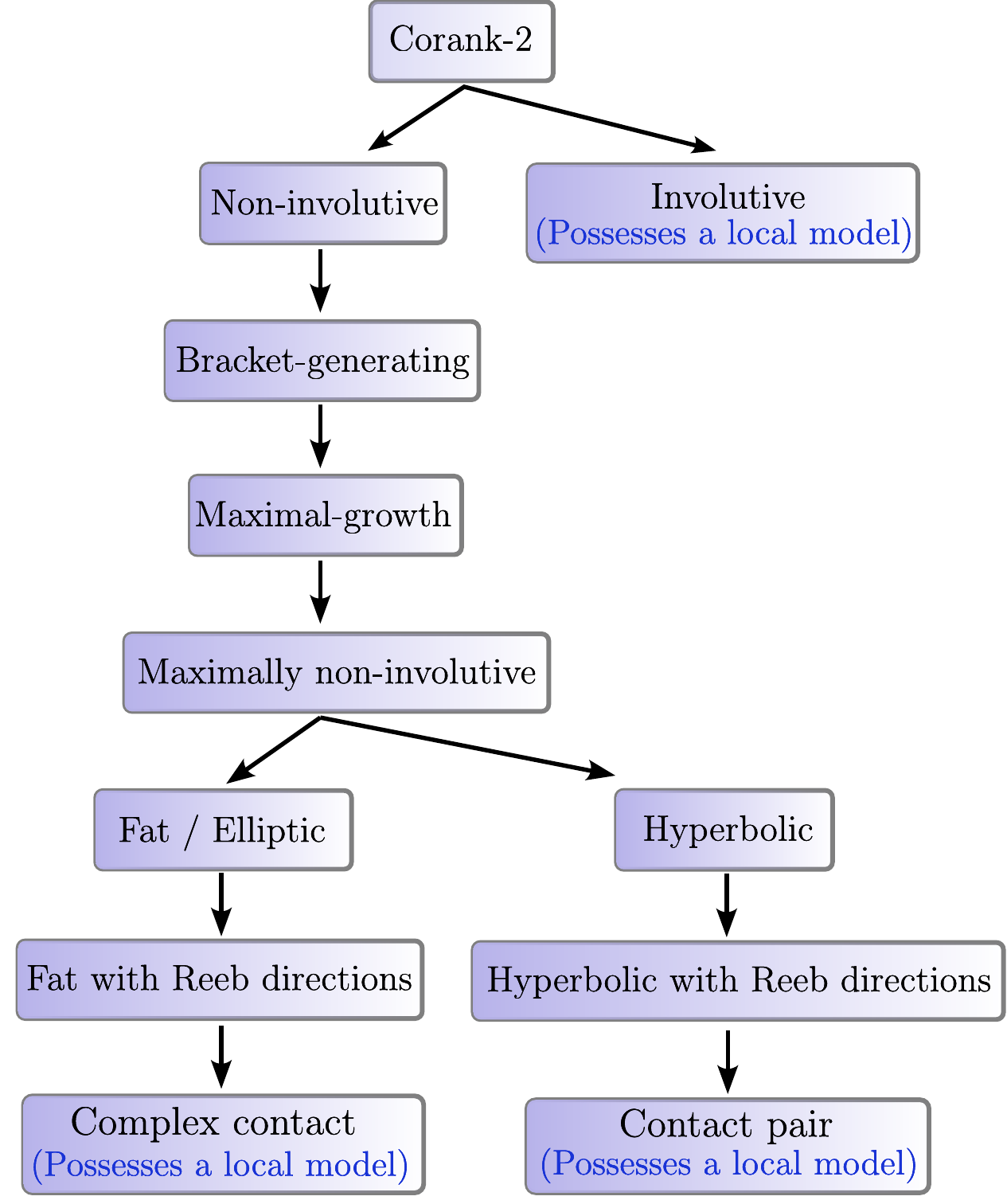}
	\caption{Hierarchy of corank-$2$ distribution-germs in dimension $6$ defined by $\operatorname{Diff}$-invariant conditions. Each vertical arrow in the tree indicates that the class on top contains the class below. By Theorem \ref{thm: existence} we know that Complex contact germs and Fat-germs with Reeb directions constitute distinct classes. Note that the subdivision is not exhaustive; i.e. we just depict the main germ-types discussed in the article.}
\label{fig:hierarchy}
\end{figure}

Figure \ref{fig:hierarchy} showcases a tree depicting different classes of $(4,6)$ distribution-germs defined by $\operatorname{Diff}$-invariant conditions, ordered by inclusion. Each vertical arrow indicates that the class on top contains the class below. Theorem \ref{thm: existence} implies that not every fat-distribution germ with Reeb directions is locally diffeomorphic to an element in the class directly below (i.e. complex contact). It would be interesting to understand analytic properties for the inclusion represented by the arrow directly above; i.e. for the space of fat distribution-germs with Reeb directions within the space of all fat distributions.
\begin{remark}
    Note that the subdivision in Figure \ref{fig:hierarchy} is not exhaustive. There are classes of germs not represented in the tree such as e.g. the class of parabolic germs or any class of distribution germs with Reeb directions that are not fat nor hyperbolic, among other classes. We just depict the main germ-types discussed in the article.
\end{remark}

It is of great interest to understand as well the analogous question that motivated this work but in the hyperbolic setting. The product of two real-contact germs, often called the \textit{flat} hyperbolic-germ in the literature (see \cite{CE2003}), represents the notion analogous to the complex contact structure germ but in the hyperbolic setting \cite{CE2003}. Note that both are called, respectively, the flat model (in the elliptic and hyperbolic settings, respectively) \cite{CE2003} and they both admit two Reeb directions (see Example \ref{ex: complex-contact} and Example \ref{ex:productcontact} below). To the best of the author's knowledge, it is not known whether the class of hyperbolic germs with Reeb directions coincides with the class of germs arising as the product of two real-contact germs. We leave this question as an interesting open question to be explored.

The question above is closely related to the local model of contact-pairs $(\alpha,\beta)$, which constitute a rich geometric structure (see the work \cite{BA} by G. Bande and A. Hadjar). Indeed, contact pairs $(\alpha,\beta)$ of $(1,1)$-type on a $6$-dimensional manifold define a distribution $\SD:=\ker(\alpha)\cap\ker(\beta)$ (called the \textit{characteristic distribution} in \cite{pairs}) satisfying that all of its germs are locally diffeomorphic to the flat hyperbolic-germ (see \cite[Theorem 3.1]{BA}). Thus, it would be interesting to understand whether this represents the local model of any hyperbolic distribution with Reeb directions or not.

Finally, it would be of interest to find further classes of fat germs defined by a $\operatorname{Diff}$-invariant condition that lie in between the classes of Fat-germs and Fat-germs with Reeb directions or in between the latter and the subclass of Complex contact germs. The analogous question is worth exploring in the hyperbolic setting as well; i.e. it would be interesting to find a class of hyperbolic germs defined by a $\operatorname{Diff}$-invariant condition that lies in between the classes of Hyperbolic germs and the class of Hyperbolic germs with Reeb directions (in case they were not the same).

\section{Geometry of corank-$2$ distributions in dimension $6$}\label{Sec:1}

We will introduce some features of the geometry and topology of smooth corank-$2$  distributions in dimension $6$. Fix a smooth $6$-manifold $M$ endowed with a rank-$4$ distribution $\SD$ and write $Q:=TM/\SD$. 

\begin{definition}
    The \textbf{curvature} (often called \textbf{Levi map}) associated to $\SD$ is the following well-defined vector bundle homomorphism:
\begin{equation*}
		\begin{array}{rccl}
		\SL \colon & \bigwedge^2 \SD & \longrightarrow &  Q \\
		& u\wedge v& \longmapsto &[u,v] \mod \SD
		\end{array}
		\end{equation*}
\end{definition}

The curvature $\SL$ measures how far is $\SD$ from being involutive. In particular, by Frobenius' Theorem, it is a foliation if and only if this morphism is zero. Otherwise, the $4$-distribution is non-involutive. If, additionally, $\SL$ is fibrewise surjective, we say that $\SD$ is of \textbf{maximal growth} (see \cite{MA2025}). On the other hand, a distribution can also be \textbf{bracket-generating}, which is a notion stronger than being non-involutive but weaker than being of maximal growth. See \cite[Sec 1.3]{MAdP2022} for a precise definition.

We can rather work with forms as well, passing to the corresponding dual picture. In particular, we can define the following bundle homomorphism which we call the dual curvature.

\begin{definition}
We call the \textbf{dual curvature} to the bundle homomorphism defined as:
\begin{equation*}
		\begin{array}{rccl}
		\omega\colon & Q^* & \longrightarrow &  \bigwedge^2\SD^* \\
		& \alpha& \longmapsto &-\alpha\circ\SL=d\alpha|_{\SD}
		\end{array}
		\end{equation*}
\end{definition}

\begin{remark}
There is a natural identification between $Q^*$ and the annihilator $\SD^\perp:=  \{\alpha\in T^*M: \alpha|_\SD =0\}$. Thus, we may rather write $\omega:\SD^\perp\to\bigwedge^2\SD^*$ instead.
\end{remark}

\begin{remark}\label{Remark:duality}
Note that $d\alpha|_\SD(X,Y) = -\alpha([X,Y])$ for any two horizontal vector fields $X, Y$, which justifies the equality $-\alpha\circ\SL=d\alpha|_{\SD}$.
\end{remark}

By identifying $\Hom(\bigwedge^2\SD,Q)$ with $(\bigwedge^2\SD^*)\otimes Q$, we can abuse notation and write $\SL\in\Gamma\left((\bigwedge^2\SD^*)\otimes Q\right)$. We can also consider the morphism $\SL\wedge \SL \in\Gamma\left(\bigwedge^4 \SD^*\otimes \Sym^2Q\right)$ which, upon fixing a volume form $\kappa\in\bigwedge^4 (\SD^*)$, defines a quadratic form $q_\SD\in\Sym^2(Q)$ on $Q^*$. More specifically, for a fixed $x\in M$ and given $\psi,\phi\in Q^*_x$, we have the quadratic form defined as:
\begin{equation}\label{qd}
    (\psi\circ\SL_x)\wedge(\phi\circ\SL_x)=q_\SD(\psi,\phi)\kappa_x.
    \end{equation}

\begin{definition}
    The distribution $(M^6, \SD^4)$ is called \textbf{maximally non-involutive} if the quadratic form $q_\SD$ is non-degenerate. Additionally, let us identify two distinguished classes that represent the generic cases \cite{dZ}. We say that:
    \begin{itemize}
        \item $(M,\SD)$ is \textbf{fat} or \textbf{elliptic} if $q_\SD$ is definite.
        \item $(M,\SD)$ is \textbf{hyperbolic} if $q_\SD$ is indefinite.
    \end{itemize}
\end{definition}

\begin{remark}
    Being fat is a more general notion not exclusive of dimension $6$. Nonetheless, in this ambient dimension being elliptic and fat are equivalent notions and we will use both terms interchangeably. See, e.g. \cite[Def. 2.2]{Bhow2024} for a more general definition of fatness.
\end{remark}

Corank-$2$ distributions of maximal growth in dimension $6$ abide by a complete $h$-principle \cite{MAdP2022, MA2025}. Nonetheless, if we restrict ourselves to the maximally non-involutive components, then some other form of analysis may be required in order to determine global homotopical properties.

The space of hyperbolic corank-$2$ distributions in dimension $6$ is known to abide by a complete $h$-principle by recent work of the author jointly with Á. del Pino by the use of a novel technique based on convex integration called \emph{Convex integration with avoidance}, \cite{MAdP2022}. Thus, the homotopical classification of such distributions reduces to their formal data; i.e. they are governed by the underlying algebro-topological information. 

The space of elliptic distributions is not so well understood and it was conjectured that it might display some rigidity at the global level \cite[Sec. 1.4.7]{MAdP2022}.  Prelegendrian submanifolds for such kind of distributions 
have been shown to yield some rigid behaviour at the global level \cite{FdPZ} and horizontal submanifolds, on the other hand, have been shown to yield
some form of flexibility \cite{Bhow2024, BhDa} at the local level. In that respect, fat $4$-distributions in dimension $6$ admitting Reeb directions yield a special subclass of distributions with significant analytic properties. Let us introduce this notion first.

\begin{definition}[\cite{Bhow2024}]\label{Def:ReebDirections}
Let $(M,\SD)$ be a corank-$2$ distribution. We say that it admits two local \textbf{Reeb directions} $Z_1, Z_2$ if $\SD$ can be locally expressed as $\SD=\ker(\lambda_1)\cap\ker(\lambda_2)$ where:
\begin{itemize}
    \item[i)] $TM=\SD\oplus\langle Z_1,Z_2\rangle$,
    
    \item[ii)] $\lambda_i(Z_j) = \delta(i,j)$ for $i,j\in\{1,2\}$. Here $\delta(i,j)$ denotes Kronecker's delta,
    
    \item[iii)] $\iota_{Z_i} d\lambda_j|_\SD = 0$ for $i,j = 1,2$ and 
    \item[iv)] $[Z_1,Z_2] = 0$.
\end{itemize}
\end{definition}

A prototypical example of a fat-distribution with Reeb directions is the real distribution underlying a complex contact structure.

\begin{example}[Complex contact germ] \label{ex: complex-contact}
    Let $(\mathbb{C}^3,\xi_{\std})$ be the standard complex contact structure on $\mathbb{C}^3$; i.e. $\xi_\std=\ker(dz-ydx)$ where $(x,y,z)$ are holomorphic coordinates. Identify $\mathbb{C}^3$ with $\R^6$ and thus write $z=z_1+i\cdot z_2$,  $x=x_1+i\cdot x_2$, $y=y_1+i\cdot y_2$. The holomorphic $1$-form $\alpha=dz-ydx$  then corresponds to $\alpha=(dz_1+i\cdot dz_2)-(y_1+i\cdot y_2)(dx_1+i\cdot dx_2)=(dz_1-y_1dx_1+y_2dx_2)+i\cdot(dz_2-y_1dx_2-y_2dx_1)$. Equivalently, this last equality can be rewritten as $\alpha=\alpha_1+i\cdot\alpha_2$, where
    \begin{align*}
    \alpha_1&=dz_1-y_1dx_1+y_2dx_2, \\
    \alpha_2&=dz_2-y_1dx_2-y_2dx_1.
    \end{align*}
    The real fat-distribution germ defined as $(\R^6,\SD:=\ker(\alpha_1)\cap\ker(\alpha_2))$ is of elliptic/fat type \cite[Example 2.3]{Bhow2024} and is often called the \textit{flat} elliptic/fat germ since it is the only germ (up to diffeomorphism) that makes the tensor S from Eq. (\ref{eq: S}) vanish \cite{CE2003}. Note that this germ admits two Reeb directions  given by $Z_1=\partial_{z_1}$ and $Z_2=\partial_{z_2}$.
\end{example}

\begin{remark}\label{Remark: Darboux}
    All complex contact manifolds of the same dimension are locally contactomorphic by the holomorphic Darboux Theorem \cite{AFL}. Therefore, it follows that all $(4,6)$-distribution germs associated to a complex contact manifold are diffeomorphic. In particular, the model from Example \ref{ex: complex-contact} is often called the holomorphic Darboux local model.
\end{remark}

Another substantially different example of a $(4,6)$-distribution germ with Reeb directions corresponds to the hyperbolic flat-model (see \cite{CE2003}) or, equivalently, the product of two real contact distribution-germs.

\begin{example}[Product of real-contact germs]\label{ex:productcontact}
Let $(\R^3,\xi_\std=\ker(dz-ydx))$ be the standard contact structure on $\R^3$. Note that the Cartesian product with itself yields the $(4,6)$-distribution $(\R^6=\R^3_{x_1,y_1,z_1}\times\R^3_{x_2,y_2,z_2}, \SD=\ker(\alpha_1)\cap\ker(\alpha_2))$, where
\begin{align*}
    \alpha_1&=dz_1-y_1dx_1\\
    \alpha_2&=dz_2-y_2dx_2.
\end{align*}
Its associated germ is known as the \textit{flat} hyperbolic-germ \cite{CE2003}. It admits two Reeb directions given by $Z_1=\partial_{z_1}$ and $Z_2=\partial_{z_2}$.
\end{example}

\begin{remark}
All real contact manifolds are locally diffeomorphic by Darboux's Theorem \cite[Thm. 2.5.1]{Geiges}. Therefore, it follows that germs arising as the product of real-contact manifolds are all diffeomorphic to each other as well (by applying suitable local diffeomorphisms on each of the factors individually).
\end{remark}

The next result, due to A. Čap and M. Eastwood, shows that every fat $(4,6)$-distribution canonically yields an associated almost complex structure. Furthermore, the obstruction for this structure to being integrable and induced by a complex contact structure is measured by a tensor $S$. 

\begin{theorem}[A. Čap \& M. Eastwood, \cite{CE2003}]\label{theorem: 1}
    Let $(M,\SD)$ be a fat distribution on an oriented $6$-dimensional manifold $M$. There exists a unique almost complex structure $J:TM\to TM$ characterised by the following properties:

\begin{itemize}
    \item[i)] $J(\SD)= \SD$; i.e. it preserves $\SD$,
    \item[ii)] $J$ induces on $M$ the same given orientation.
    \item[iii)] The curvature $\SL:\SD\times\SD\to Q$ is complex bilinear for the induced structures; i.e. $
    [u,v]+J[Ju,v]\in\Gamma(\SD) \text{ for every }u,v\in\Gamma(\SD).$
    \item[iv)] $[u,v]+J[Ju,v]-J[u,Jv]+[Ju,Jv]\in\Gamma(\SD)$ for every $u\in\Gamma(TM)$ and every $v\in\Gamma(\SD)$.
\end{itemize}
Additionally, $J$ is integrable if and only if the tensor $S:Q\otimes \SD\to Q$ induced by the following morphism vanishes:
\begin{equation}\label{eq: S}
    S(u,v)=[u,v]+J[Ju,v] \mod \SD \quad \text{ for }u\in\Gamma(TM)\text{ and } v\in\Gamma(\SD).
\end{equation}
    
\end{theorem}

The next result is a direct consequence of Theorem \ref{theorem: 1} (see \cite{CE2003}).

\begin{theorem}\label{thm:LocalComplexContact}
    Let $(\mathcal{U},\SD)$ be the germ of a fat distribution in dimension $6$. Then, the tensor $S$ is the only obstruction for the germ to be induced by a complex contact structure; i.e.  $(\mathcal{U},\SD)$ is induced by a complex contact structure if and only if the tensor $S$ vanishes.
    \end{theorem}
   
\subsubsection{Description of the almost complex structure $J:TM\to TM$.}\label{subsection:J}
Let us describe how the almost complex structure $J:TM\to TM$ in Theorem \ref{theorem: 1} is defined.  This will be relevant for further discussions. We follow the exposition from \cite[pp. 94-95]{CE2003}.

Since $(M,\SD)$ is fat, then $q_\SD$ is definite. Equivalently, for fixed $x\in M$, there does not exist any non-zero $\psi\in Q_x^*$ such that  $q_\SD(\psi,\psi)=(\psi\circ\SL_x)\wedge(\psi\circ\SL_x)$ vanishes. This means that if we regard $q_\SD$ as a quadratic polynomial, it does not possess real roots. Nonetheless, it does possess two complex (conjugate) roots (well defined up to complex scale). In other words, there are two (conjugate) elements $\psi_1,\psi_2\in Q^*_x\otimes\mathbb{C}$ such that $q_\SD(\psi_i,\psi_i)=(\psi_i\circ\SL_x)\wedge(\psi_i\circ\SL_x)$ is the zero element in $\bigwedge^4(\SD^*)\otimes\mathbb{C}$ for $i=1,2$. By the Plücker criterion, $\psi_i\circ\SL\in\bigwedge^2\SD^*\otimes\mathbb{C}$ are simple as $2$-forms. Take one of the two roots $\psi\in Q^*_x\otimes\mathbb{C}$ and note the following two facts:
\begin{itemize}
    \item[i)] Since $\psi\in Q^*_x\otimes\mathbb{C}$ yields a simple $2$-form $\psi\circ\SL\in\bigwedge^2\SD^*\otimes\mathbb{C}$, it then defines a complex $2$-plane within $\SD^*\otimes\mathbb{C}$. In turn, this complex plane defines a complex structure $J:\SD_x\to\SD_x$. See Subsection \ref{JD} for further details and some example of application.

    \item[ii)] On the other hand, $\psi\in Q^*_x\otimes\mathbb{C}$  readily identifies $Q_x$ with the complex numbers $\mathbb{C}$. Thus, this identification endows $Q_x$ with a complex structure $J:Q_x\to Q_x$ in the obvious way; i.e. $J:{Q_x}\to Q_x,\ v\mapsto\psi^{-1}\left(i\cdot(\psi(v))\right)$.
\end{itemize}

Moreover, note that by construction these complex structures do not change under complex multiplication of $\psi$. They are thus uniquely determined once one of the two (conjugate) roots has been chosen. The other root yields $-J$ instead of $J$ but, since $M$ was oriented from scratch, there is thus a unique canonically defined $J$ which is the one that induces the orientation of $M$.
\begin{remark}\label{signJ}
    Note that, by the bilinearity of the Lie bracket, both $J$ and $-J$ induce the same tensor (see Eq. (\ref{eq: S})).
\end{remark}
This construction thus yields  complex structures on $\SD_x$ and $Q_x$  inducing the orientation of $M$ such that
\begin{equation}\label{ec:EcuacionClave}
    \SL(u,v)+J\SL(Ju,v)=0\quad\text{for }u,v\in\Gamma(\SD).
\end{equation}

Choose now any extension of these almost complex structures to an almost complex structure $\tilde{J}:TM\to TM$ and define the tensor $\tilde{S}:Q\otimes\SD\to Q$ induced by:

    \begin{equation}\label{eq: Stilde}
    \tilde{S}(u,v)=[u,v]+ J[\tilde{J}u,v] \mod \SD \quad \text{ for }u\in\Gamma(TM)\text{ and } v\in\Gamma(\SD).
\end{equation}

As noted in \cite{CE2003}, $\tilde{J}$ satisfies the first three conditions $i), ii), iii)$ from the statement of Theorem \ref{theorem: 1}. Nonetheless, \cite{CE2003} also notes that $\tilde{S}$ depends on the chosen extension $\tilde{J}$ and, so, we will proceed as follows in order to eliminate that dependence and thus make the construction canonical. For a fixed choice $u\in TM$, consider the following complex linear map:

\begin{equation*}
		\begin{array}{rccl}
		h\colon & \SD & \longrightarrow &  Q \\
		& v& \longmapsto & \dfrac{-\tilde{S}(u,v)+J\tilde{S}(u,Jv)}{2}
		\end{array}
		\end{equation*}

Since $\SL$ is non-degenerate, there exists a unique element $Ku\in\SD$ so that:
\begin{equation}\label{eq:LKuv}
    \SL(Ku,v)=\frac{-\tilde{S}(u,v)+J\tilde{S}(u,Jv)}{2}\quad\text{ for }u\in\Gamma(TM)\text{ and }v\in\Gamma(\SD).
\end{equation}

Therefore, we have now defined the homomorphism $K:TM\to \SD$. And, finally, $J=\tilde{J}+K$ is the almost complex structure $J:TM\to TM$ from Theorem \ref{theorem: 1} which is characterised by conditions $i)$ to $iv)$. Note that we have merely limited ourselves to describing its construction but the interested reader may check \cite{CE2003} for a proof and further details.

\section{A fat distribution with Reeb directions that is not complex contact}\label{sec:3}

The present section is devoted to the construction of a fat $(4,6)$-distribution with two Reeb directions that does not support a complex contact structure anywhere, not even locally nor up to diffeomorphism. We first state Theorem \ref{mainthm}, which does the work since it provides such an example in the region $\{|x_2|<1\}$ of $\R^6$. We will break its proof into several lemmas and propositions that we will state and prove all along the section. 

Finally, we will show how this construction can be globalised; i.e. we can find a global distribution on the whole $\R^6$ satisfying conditions $i), ii)$ and $iii)$ as in Theorem \ref{mainthm}. This will be the content of Theorem \ref{thm:global}.

\begin{theorem}\label{mainthm}
    Consider the distribution $\left(\mathbb{R}^6, \SD=\ker(\lambda_1)\cap\ker(\lambda_2)\right)$ defined by the following smooth $1$-forms:
    \begin{align*}
        \lambda_1&=dz_1-y_1dx_1-y_2dx_2-\left(\frac{x_2^3}{3}+x_2+2x_1\right)dy_1\\
        \lambda_2&=dz_2-y_2dx_1-y_1dx_2.
    \end{align*}
    The associated distribution germ around any point $q=(x_1,x_2,y_1,y_2,z_1,z_2)$ in the region $\{|x_2|<1\}\subset\R^6$ satisfies the following properties: 
    \begin{itemize}
        \item[i)] It is a corank-$2$ fat distribution germ.
        \item[ii)] It admits two Reeb directions given by $X_1=\partial_{z_1}$ and $X_2=\partial_{z_2}$.
        \item[iii)] Any diffeomorphic germ does not support a complex contact structure.
    \end{itemize}
\end{theorem}

Henceforth, we fix the distribution $(\R^6,\SD=\ker(\lambda_1)\cap\ker(\lambda_2))$ for the rest of the article. Whenever we write $\SD$ we will refer to the specific distribution from Theorem \ref{mainthm} unless otherwise stated. The proof of Theorem \ref{mainthm} is delayed until the end of the section. Let us start proving that it satisfies the fatness condition in an open region of $\R^6$.

\begin{lemma}\label{lemma:fatness}
   The distribution $\left(\mathbb{R}^6, \SD=\ker(\lambda_1)\cap\ker(\lambda_2)\right)$ 
    is fat in the region $\{|x_2|<1\}$.
\end{lemma}
\begin{proof}
First, note that if we write $A:= \dfrac{-x_2^3}{3}-x_2-2x_1$ and we denote $A_{x_i}:=\frac{\partial}{\partial x_i}(A)$, $i=1,2$, then we have: 
\begin{align}
    d\lambda_1& = -dy_1\wedge dx_1 - dy_2\wedge dx_2 + dA\wedge dy_1 =    dx_1\wedge dy_1 + dx_2\wedge dy_2 + A_{x_1}dx_1\wedge dy_1 + A_{x_2}dx_2\wedge dy_1 = \label{dlambda1} 
    \\ & = (A_{x_1}+1) dx_1\wedge dy_1 + A_{x_2} dx_2\wedge dy_1 + dx_2\wedge dy_2,\notag \\
    d\lambda_2&=dx_1\wedge dy_2+dx_2\wedge dy_1. \label{dlambda2} 
\end{align}

Therefore, this yields the following:
\begin{align*}
&(d\lambda_1)^2= -2(A_{x_1}+1) dx_1\wedge dx_2\wedge dy_1\wedge dy_2, \\ 
 &   (d\lambda_2)^2 = 2 dx_1\wedge dx_2 \wedge dy_1\wedge dy_2,\\
 & d\lambda_1\wedge d\lambda_2 = A_{x_2}dx_1\wedge dx_2 \wedge dy_1\wedge dy_2. \label{dlambda2}
\end{align*}

And so, if we fix the volume form  $\tau|_\SD:=dx_1\wedge dx_2\wedge dy_1\wedge dy_2$ in $\SD$, the quadratic form from Eq. (\ref{qd}) has the following associated matrix:

\begin{equation}
    M_{q_\SD}=\begin{pmatrix}
(d\lambda_1)^2 & (d\lambda_1\wedge d\lambda_2) \\
(d\lambda_1\wedge d\lambda_2) & (d\lambda_2)^2
\end{pmatrix}=\begin{pmatrix}
-2(A_{x_1}+1) & A_{x_2} \\
A_{x_2} & 2
\end{pmatrix} = \begin{pmatrix}
2 & -x_2^2-1 \\
-x_2^2-1 & 2
\end{pmatrix}.
\end{equation}

Finally, $\det(M_{q_\SD})=4-(x_2^2+1)^2$ and this expression is positive whenever $|x_2^2+1|<2$ or equivalently, whenever $|x_2|<1$. Thus, $\SD$ is fat or elliptic in the region $\{|x_2| <1\}\subset\R^6$.\end{proof}

\begin{remark}\label{x2smallerthan1}Henceforth, unless otherwise stated, we will always assume that we are working in the region $\{|x_2| <1\}\subset\R^6$ where our distribution satisfies the fatness condition. This will be implicit in all the subsequent calculations.
\end{remark}

The following lemma provides a global framing for the distribution $(\R^6, \SD)$.

\begin{lemma}\label{lemma:framing} A global framing of the distribution $\left(\mathbb{R}^6, \SD=\ker(\lambda_1)\cap\ker(\lambda_2)\right)$ defined by:
    \begin{align*}
        \lambda_1&=dz_1-y_1dx_1-y_2dx_2-\left(\frac{x_2^3}{3}+x_2+2x_1\right)dy_1\\
        \lambda_2&=dz_2-y_2dx_1-y_1dx_2
    \end{align*}
    
is given by $\{\Xi_1,\Xi_2,\Xi_3,\Xi_4\}$, where:

\begin{align*}
    \Xi_1&=\partial_{x_1}+y_1\partial_{z_1}+y_2\partial_{z_2}, &
    \Xi_2&= \partial_{x_2}+y_2\partial_{z_1}+y_1\partial_{z_2},\\
    \Xi_3&=\partial_{y_1}+\left(\frac{x_2^3}{3}+x_2+2x_1\right)\partial_{z_1}, &
    \Xi_4&=\partial_{y_2}.
\end{align*}
\end{lemma}
\begin{proof}
    It follows from the fact that $\lambda_i(\Xi_j)=0$ for all $i\in\{1,2\}$ and $j\in\{1,2,3,4\}$ together with the fact that $\{\Xi_1,\Xi_2,\Xi_3,\Xi_4\}$ is a linearly independent set of vectors.
\end{proof}

The next result states that $(\R^6,\SD)$ admits two Reeb directions and provides  their analytic description.

\begin{lemma}\label{lemma:reeb}
   The distribution $\left(\mathbb{R}^6, \SD=\ker(\lambda_1)\cap\ker(\lambda_2)\right)$ admits two Reeb directions $v_1,v_2$ given by $v_1=\partial_{z_1}$ and $v_2=\partial_{z_2}$.
   \end{lemma}

\begin{proof}
It suffices to check that the four conditions from Definition \ref{Def:ReebDirections} are satisfied. Note that condition $i)$ readily follows by observing that the vectors in the global framing of $(\R^6,\SD)$ given by Lemma \ref{lemma:framing} together with the coordinate vectors $\partial_{z_1}$ and $\partial_{z_2}$ span the whole tangent bundle $T\R^6$. Conditions $ii)$ and $iii)$ follow from the analytic expression of $\lambda_1$ and $\lambda_2$. Finally, condition $iv)$ is immediate since $\partial_{z_1}$ and $\partial_{z_2}$ are both coordinate directions in $T\R^6$ and they thus commute. This yields the claim.
\end{proof}
Following the procedure described along Subsection \ref{subsection:J}, we want to find the two roots $\psi_1, \psi_2\in Q_x^*\otimes\mathbb{C}$ for which
$q_\SD(\psi_i,\psi_i)=(\psi_i\circ\SL_x)\wedge(\psi_i\circ\SL_x)$ is the zero element in $\bigwedge^4(\SD^*)\otimes\mathbb{C}$. Since they are conjugate to each other, it suffices to find one of them. This is the content of the following lemma.

\begin{lemma}\label{election:t}
Let $t:=\dfrac{x_2^2+1}{2}+ i\left(\frac{\sqrt{3-2x_2^2-x_2^4}}{2}\right)\in\mathbb{C}$. The $1$-form $\psi:= \lambda_1+t\cdot\lambda_2$ (as well as its complex conjugate) satisfies that $q_\SD(\psi,\psi)=(\psi\circ\SL_x)\wedge(\psi\circ\SL_x)$ is the zero element in $\bigwedge^4(\SD^*)\otimes\mathbb{C}$; i.e. $\psi$ is a complex root of the quadratic form $q_\SD$.
\end{lemma}

\begin{proof}
Write $\psi=\lambda_1+s\cdot\lambda_2$ as $(1,s)$ in the basis $\{\lambda_1,\lambda_2\}$ of $Q^*$. Then, we can just check that 

\begin{equation*}
    \begin{pmatrix}1 & s\end{pmatrix}\begin{pmatrix}
2 & -x_2^2-1 \\
-x_2^2-1 & 2
\end{pmatrix}\begin{pmatrix}
    1\\
    s
\end{pmatrix}= 0\Longleftrightarrow 2s^2-2s(x_2^2+1)+2= 0\Longleftrightarrow s^2-s(x_2^2+1)+1=0.
\end{equation*}

And, so, the solutions to this equation are:

\[
s=\dfrac{x_2^2+1}{2}\pm i\left(\frac{\sqrt{4-(x_2^2+1)^2}}{2}\right)=\dfrac{x_2^2+1}{2}\pm i\left(\frac{\sqrt{3-2x_2^2-x_2^4}}{2}\right).
\] Note that we used the fact that $|x_2|<1$ (Remark \ref{x2smallerthan1}); otherwise the expression $4-(x_2^2+1)^2$ within the square root might be negative. This concludes the proof of the lemma.\end{proof}

We will henceforth reserve the letter $t$ in order to denote the complex number $$t=\dfrac{x_2^2+1}{2}+ i\left(\frac{\sqrt{3-2x_2^2-x_2^4}}{2}\right).$$

\begin{remark}
Strictly speaking, $t$ is not a fixed complex number but a smooth function $t:(-1,1)_{x_2}\to\mathbb{C}$ that depends on $x_2\in(-1,1)$. We will henceforth treat it as such although we will abuse terminology and refer to it as a complex number and write simply $t\in\mathbb{C}$ in order not to overload the notation. 
\end{remark}

\subsection{Almost complex structure $J:Q\to Q$ on $Q=TM/\SD$}

Recall that $\psi\in Q^*_x\otimes\mathbb{C}$ allows to identify $Q_x$ with the complex numbers $\mathbb{C}$. Therefore, this identification endows $Q_x$ with a complex structure $J:Q_x\to Q_x$. We will now compute such $J$.

\begin{remark}\label{ConjugateSign}
    As noted before, we could choose $\bar{\psi}$ (the complex conjugate form of $\psi$) instead of $\psi$ from scratch. They determine opposite almost complex structures $J, -{J}$. Nonetheless, since the ultimate goal of this article is to show that the tensor $S$ from Eq. (\ref{eq: S}) does not vanish, in view of Remark \ref{signJ}, we could choose either $\psi$ or $\bar{\psi}$ with identical outcome. We will thus henceforth work with $\psi$ from Lemma \ref{election:t} without loss of generality.
\end{remark}

Since the Reeb directions $v_1=\partial_{z_1}$, $v_2=\partial_{z_2}$ provide the splitting $TM=\SD \oplus \langle v_1,v_2\rangle$, we can actually express any class $[w]\in Q=TM/\SD$ by a representative of the form $w=av_1+bv_2\in \langle v_1,v_2\rangle$, where $a,b\in\R$. Since we also have a well defined $J:\SD_x\to\SD_x$ over each point $x\in M$, this allows to define an extension $\tilde{J}:T_xM\to T_xM$ in a straightforward fashion; i.e. by considering the splitting $TM=\SD \oplus \langle v_1,v_2\rangle$ and extending by linearity. Let us thus compute $J(av_1+bv_2)$ explicitly by using the $1$-form $\psi\in Q_x^*\otimes\mathbb{C}$. First note that
\begin{align*}
    \psi(av_1+bv_2)&= a\psi(v_1)+b\psi(v_2)=a(\lambda_1(v_1)+t\underbrace{\lambda_2(v_1)}_{0})+b(\underbrace{\lambda_1(v_2)}_0+t\lambda_2(v_2))=\\
    &=a+bt=(a+\operatorname{Re}(t)\cdot b)+i\cdot(\operatorname{Im}(t)\cdot b).
\end{align*}

Note that we have used condition $ii)$ from Definition \ref{Def:ReebDirections} in the third equality. We thus have that the matrices associated to $\psi:Q\to\mathbb{C}$ and its inverse $\psi^{-1}:\mathbb{C}\to Q$, when expressed with respect to the bases $\{v_1,v_2\}$ and $\{1,i\}$, are:

\[M_\psi=
\begin{pmatrix}
1 & \operatorname{Re}(t) \\
0 & \operatorname{Im}(t)
\end{pmatrix},\quad 
M_{\psi^{-1}}=
\begin{pmatrix}
1 & -\frac{\operatorname{Re}(t)}{\operatorname{Im}(t)} \\
0 & \frac{1}{\operatorname{Im}(t)}
\end{pmatrix}.
\]

We thus have that $i\cdot\psi(av_1+bv_2)=-b\cdot \operatorname{Im}(t)+i\cdot(a+\operatorname{Re}(t)\cdot b)$. In other words, if we regard $i\cdot\psi:Q\to\mathbb{C}$ as a linear map, then its associated matrix with respect to the bases $\{v_1,v_2\}$ and $\{1,i\}$ is
\[M_{i \cdot\psi}=
\begin{pmatrix}
0 & -\operatorname{Im}(t) \\
1 & \operatorname{Re}(t)
\end{pmatrix}.
\]

We can finally describe the almost complex structure as $J(av_1+bv_2)=\psi^{-1}\circ (i\cdot\psi)(av_1+bv_2)$. Its associated matrix with respect to the basis $\{v_1,v_2\}$ is thus
\[M_{J}=
\begin{pmatrix}
1 & -\frac{\operatorname{Re}(t)}{\operatorname{Im}(t)} \\
0 & \frac{1}{\operatorname{Im}(t)}
\end{pmatrix}\cdot\begin{pmatrix}
0 & -\operatorname{Im}(t) \\
1 & \operatorname{Re}(t)
\end{pmatrix}=\begin{pmatrix}
\frac{-\operatorname{Re}(t)}{\operatorname{Im}(t)} & -\operatorname{Im}(t)-\frac{\operatorname{Re}(t)^2}{\operatorname{Im}(t)} \\
\frac{1}{\operatorname{Im}(t)} & \frac{\operatorname{Re}(t)}{\operatorname{Im}(t)}
\end{pmatrix}.\]

For our particular choice $t=\frac{x_2^2+1}{2}+i\cdot\frac{\sqrt{3-2x_2^2-x_2^4}}{2}\in\mathbb{C}$, we have that

\[
\operatorname{Re}(t)=\frac{x_2^2+1}{2},\quad \operatorname{Im}(t)=\frac{\sqrt{3-2x_2^2-x_2^4}}{2},
\]
\[
\operatorname{Re}(t)^2=\frac{(x_2^2+1)^2}{4},\quad  \frac{\operatorname{Re}(t)}{\operatorname{Im}(t)}=\frac{x_2^2+1}{\sqrt{3-2x_2^2-x_2^4}},\quad \frac{\operatorname{Re}(t)^2  }{\operatorname{Im}(t)}=\frac{(x_2^2+1)^2}{2\sqrt{3-2x_2^2-x_2^4}}
\]

and thus we get the following associated matrix:

\[M_{J}=\frac{1}{\sqrt{3-2x_2^2-x_2^4}}
\begin{pmatrix}
-(x_2^2+1) & -2 \\
2 & x_2^2+1
\end{pmatrix}.\]
This way, we can compute $J(v_1)=J\partial_{z_1}$ and $J(v_2)=J\partial_{z_2}$:
\begin{align}\label{Jv1}
    J\partial_{z_1}&=M_J\cdot\begin{pmatrix}
1 \\
0
\end{pmatrix}=\frac{-(x_2^2+1)}{\sqrt{3-2x_2^2-x_2^4}}\cdot\partial_{z_1}+\frac{2}{\sqrt{3-2x_2^2-x_2^4}}\partial_{z_2}.\\
    J\partial_{z_2}&=M_J\cdot\begin{pmatrix}
0 \\
1
\end{pmatrix}=\frac{-2}{\sqrt{3-2x_2^2-x_2^4}}\cdot\partial_{z_1}+\frac{x_2^2+1}{\sqrt{3-2x_2^2-x_2^4}}\partial_{z_2}.\label{Jv2}
\end{align}

Since $\{v_1,v_2\}$ is a basis of $Q$, this allows to compute $J(av_1+bv_2)$ for any other combination.

\subsection{Almost complex structure on $\SD$}\label{JD}

Let us briefly discuss how a decomposition $\SD\otimes\mathbb{C}=W\oplus\bar{W}$ of $\SD$ into two complex planes such that $W\cap\bar{W}=\{0\}$ canonically defines a complex structure $J:\SD\to\SD$ and viceversa. Let us start with the latter part.

If $J:\SD\to\SD$ is a complex structure, then $\SD\otimes\mathbb{C}=W\oplus\bar{W}$, where
\begin{align*}W=\left(\SD\otimes\mathbb{C}\right)^-,\quad \bar{W}=\left(\SD\otimes\mathbb{C}\right)^+.
\end{align*}

Here $W$ and $\bar{W}$ denote the eigenspaces of $J$ associated to the eigenvalues $-i$ and $+i$, respectively. Conversely, the aforementioned decomposition $\SD\otimes\mathbb{C}=W\oplus\bar{W}$ canonically defines a complex structure $J:\SD\to\SD$ as follows. Declare $J|_W=-i\cdot Id$, $J|_{\bar{W}}=i\cdot\Id$ and extend by linearity. Then we get a complex linear endomorphism on $\SD\otimes\mathbb{C}$ which, by restriction to $\SD\subset\SD\otimes\mathbb{C}$, yields a complex structure $J:\SD\to\SD$. We are considering the natural inclusion $\SD\subset \SD\otimes\mathbb{C}$ here, which is given by $w\mapsto w\otimes 1$.

In the particular case of our work, a decomposition of $\SD\otimes\mathbb{C}=W\oplus\bar{W}$ is naturally induced by the complex plane $W=\ker(d\psi|_\SD)$. This way, we can  recover the almost complex structure considered by A. Čap and M. Eastwood in \cite{CE2003} by identifying $W$ with its $(-i)$-eigenspace $(\SD\otimes\mathbb{C})^-$.

\begin{remark}
    The fact that the identification of the subspace $\ker(d\psi|_\SD)$ is made with the $(-i)$-eigenspace $(\SD\otimes\mathbb{C})^-$ of $J|_\SD$, and not with the $(+i)$-eigenspace $(\SD\otimes\mathbb{C})^+$ instead, is relevant. The latter would not yield the almost complex structure considered in \cite{CE2003}. Indeed, note that Eq. (\ref{ec:EcuacionClave}) implies that $\SL(Ju,v)=J\SL(u,v)$ for $u,v\in\SD$. By applying $\psi:Q\to \mathbb{C}$ to both sides of the expression, we get
    \[\psi(\SL(Ju,v))=\psi(J\SL(u,v))=i\cdot\psi(\SL(u,v))\] or, equivalently, $d\psi|_\SD(Ju,v)=i\cdot d\psi|_\SD(u,v)$ (Remark \ref{Remark:duality}). Therefore, in view of this equality, $\ker(d\psi|_\SD)$ contains the $(-i)$-eigenspace  $(\SD\otimes\mathbb{C})^-$ of the almost complex structure $J:\SD\to\SD$ from \cite{CE2003}. Now just note that a simple non-zero $2$-form $\psi\circ\SL\in\bigwedge^2\SD^*\otimes\mathbb{C}$ has a complex $2$-dimensional kernel for dimensional reasons. Since $(\SD\otimes\mathbb{C})^-$ also has complex dimension $2$, we conclude that this last subspace inclusion was in fact an equality; i.e. $\ker(d\psi|_\SD)=(\SD\otimes\mathbb{C})^-$.
\end{remark}

Let us find first which complex plane $\Pi\subset\SD^*\otimes\mathbb{C}$ does the $1$-form $\psi\in Q^*\otimes\mathbb{C}$ determine, where recall that $$\psi=\lambda_1+t\cdot\lambda_2=dz_1+tdz_2-(y_1+ty_2)dx_1-(y_2+ty_1)dx_2-\left(\frac{x_2^3}{3}+x_2+2x_1\right)dy_1.$$
Note that
\begin{align} d\psi|_\SD&=d(\lambda_1+t\lambda_2)|_\SD=(d\lambda_1+\underbrace{dt\wedge\lambda_2}_{=0\text{ in }\SD}+td\lambda_2)|_\SD=\\&=-dy_1\wedge dx_1-dy_2\wedge dx_2-(x_2^2+1)dx_2\wedge dy_1-2dx_1\wedge dy_1 - tdy_2\wedge dx_1 - tdy_1\wedge dx_2=\\
    &=-dx_1\wedge dy_1 + dx_2\wedge dy_2+(t-x_2^2-1)dx_2\wedge dy_1 +tdx_1\wedge dy_2. \label{expression}
\end{align}

Recall that, according to the explanation from Subsection \ref{subsection:J}, the $2$-form $d\psi|_\SD$ must be simple by Plücker's criterion. The next lemma shows its explicit decomposition as the wedge product of two $1$-forms.

\begin{lemma}\label{lemma:dpsi}
The $2$-form $d\psi|_\SD\in\bigwedge^2\SD^*\otimes\mathbb{C}$ admits the  complex decomposition $d\psi|_\SD= \alpha\wedge \beta,$ where
\begin{align*}
\alpha&=dx_1-(t-x_2^2-1)dx_2,\\
\beta&=-dy_1+tdy_2. 
\end{align*} In particular, this implies that $d\psi|_\SD\in\bigwedge^2\SD^*\otimes\mathbb{C}$ is a simple $2$-form.
\begin{proof}
    Expanding the expression $\alpha\wedge\beta$, we get that 
    \begin{align*}
&\left(dx_1-(t-x_2^2-1)dx_2\right)\wedge
\left(-dy_1+tdy_2\right)=\\
&-dx_1\wedge dy_1+tdx_1\wedge dy_2+(t-x_2^2-1)dx_2\wedge dy_1 -t(t-x_2^2-1)dx_2\wedge dy_2.
    \end{align*}
Comparing this expression with the one in Eq. (\ref{expression}), it suffices to show that $-t(t-x_2^2-1)= 1$ in order to conclude.
Note that this is equivalent to showing that $t^2-t(x_2^2+1)+1=0$. Notice that since $|x_2|<1$, as we are assuming all over the article (Remark \ref{x2smallerthan1}), the equation $s^2-s(x_2^2+1)+1=0$ holds true if and only if 
$$
s=\frac{(x_2^2+1)\pm\sqrt{(x_2^2+1)^2-4}}{2}=\dfrac{x_2^2+1}{2}\pm i \frac{\sqrt{3-2x_2^2-x_2^4}}{2}.
$$
Therefore, from our choice of $t\in\mathbb{C}$ (recall Lemma \ref{election:t}), the claim readily follows.
\end{proof}
\end{lemma}

The $2$-form $d\psi|_\SD=\alpha\wedge\beta\in\bigwedge^2\SD^*\otimes\mathbb{C}$ thus defines the complex plane $W:=\ker(\alpha)\cap\ker(\beta)\subset\SD\otimes\mathbb{C}$. In particular, it yields the following decomposition
\begin{equation}
    \SD\otimes\mathbb{C}= W\oplus\bar{W},
\end{equation}
where the vectors 
\begin{align}\label{u1}
    u_1&=(\partial_{x_2}+y_2\partial_{z_1}+y_1\partial_{z_2}) + (t-x_2^2-1)\cdot(\partial_{x_1}+y_1\partial_{z_1}+y_2\partial_{z_2}),\\
    u_2&= \partial_{y_2}+t\cdot(\partial_{y_1}+(x_2^3/3+x_2+2x_1)\partial_{z_1})\label{u2}
\end{align}
span the complex plane $W=\langle u_1,u_2\rangle$ and their conjugate vectors $\bar{u}_1,\bar{u}_2$ span the corresponding complex plane $\bar{W}=\langle \bar{u}_1,\bar{u}_2\rangle$. Indeed, note that both vectors $u_1,u_2$ are elements in $\SD\otimes\mathbb{C}$; i.e. $u_1,u_2\in\SD\otimes\mathbb{C}$ and they both lie in $\ker(\alpha)\cap\ker(\beta)$.

We can now use the decomposition above in order to compute $\tilde{J}$. In particular, we will compute $\tilde{J}\nu$ for a particular choice of $\nu\in\SD$ in the next Lemma.

\begin{remark}\label{Remark:JandJtilde}
    The reader should remember that $\tilde{J}:TM\to TM$ is an extension both of $J:Q\to Q$ and of $J:\SD\to\SD$. Therefore, we may write either $J$ or $\tilde{J}$ acting on vectors in $Q$ interchangeably since both endomorphisms agree on $Q$.
\end{remark}

\begin{lemma}\label{lemma: Jv}
    Take $\nu=\partial_{x_2}+y_2\partial_{z_1}+y_1\partial_{z_2}\in\SD$. Then: $$\tilde{J}\nu=-\left(\frac{x_2^2+1}{\sqrt{3-2x_2^2-x_2^4}}\right)\left(\partial_{x_2}+y_2\partial_{z_1}+y_1\partial_{z_2}\right)+\left(\frac{2}{\sqrt{3-2x_2^2-x_2^4}}\right)\left(\partial_{x_1}+y_1\partial_{z_1}+y_2\partial_{z_2}\right).$$
\end{lemma}
\begin{proof}
First, we decompose $\nu$ as the sum of an element $\omega\in W$ and its conjugate $\bar{\omega}\in\bar{W}$:
\begin{equation}\label{decompositionv}
    \nu=\underbrace{\left(\frac{1}{2}-i\cdot\frac{x_2^2+1}{2\sqrt{3-2x_2^2-x_2^4}}\right)u_1}_\omega+\underbrace{\left(\frac{1}{2}+i\cdot\frac{x_2^2+1}{2\sqrt{3-2x_2^2-x_2^4}}\right)\bar{u}_1}_{\bar{\omega}},
\end{equation}

where recall that $u_1$ and $u_2$ are defined in expressions (\ref{u1}) and (\ref{u2}). 

In order to check that the decomposition from Eq. (\ref{decompositionv}) is correct, it suffices to check that $\nu=2\cdot\Re(\omega)$, where $\Re(\omega)$ denotes the real part of $\omega$. On one hand, we have that

\begin{align}
\omega &= \left(\frac{1}{2}-i\cdot\frac{x_2^2+1}{2\sqrt{3-2x_2^2-x_2^4}}\right)\cdot 
\left((\partial_{x_2}+y_2\partial_{z_1}+y_1\partial_{z_2}) 
+ (t-x_2^2-1)\cdot(\partial_{x_1}+y_1\partial_{z_1}+y_2\partial_{z_2})\right)
\label{eqomega1}\\
&=\left(\frac{1}{2}-i\cdot\frac{x_2^2+1}{2\sqrt{3-2x_2^2-x_2^4}}\right)\cdot 
\left(\partial_{x_2}+y_2\partial_{z_1}+y_1\partial_{z_2}\right) \notag\\ 
&+\left(\frac{1}{2}-i\cdot\frac{x_2^2+1}{2\sqrt{3-2x_2^2-x_2^4}}\right)\cdot 
\left(\dfrac{-x_2^2-1+ i \sqrt{3-2x_2^2-x_2^4}}{2}\right)\cdot
(\partial_{x_1}+y_1\partial_{z_1}+y_2\partial_{z_2})\notag
\end{align}

And, therefore,
\begin{align*}
    \Re(\omega)&=\frac{1}{2}\cdot (\partial_{x_2}+y_2\partial_{z_1}+y_1\partial_{z_2}) +\frac{-x_2^2-1+x_2^2+1}{4}\cdot(\partial_{x_1}+y_1\partial_{z_1}+y_2\partial_{z_2}),
    \end{align*}
from where we conclude that $2\cdot\Re(\omega)=\partial_{x_2}+y_2\partial_{z_1}+y_1\partial_{z_2}=\nu$. 
So, now that we have checked that the decomposition from Eq. (\ref{decompositionv}) is correct, we just have that \begin{align*}
    \tilde{J}\nu=-i\omega+i\bar{\omega}=-\frac{x_2^2+1}{\sqrt{3-2x_2^2-x_2^4}}\left(\partial_{x_2}+y_2\partial_{z_1}+y_1\partial_{z_2}\right)+\frac{2}{\sqrt{3-2x_2^2-x_2^4}}\left(\partial_{x_1}+y_1\partial_{z_1}+y_2\partial_{z_2}\right).
\end{align*}
The last equality follows by taking into account that $-i\omega+i\bar{\omega}=2\Im(\omega)$, where $\Im(\omega)$ denotes the imaginary part of $\omega$, which can be read directly from the second equality in Eq. (\ref{eqomega1}). This yields the claim.\end{proof}
Our ultimate goal is to show that the tensor $S$ (Eq. (\ref{eq: S})) does not vanish. Therefore, if we find two vectors $R\in TM$, $\nu\in\SD$ so that $S(R,\nu)\neq 0 \mod \SD$, we will be done. For the ease of notation, we will denote by $u \equiv v$ any equality of the form $u=v\mod\SD$; i.e. it denotes an equality in the quotient $Q=TM/\SD$. We make the choices $R=\partial_{z_2}$ and $\nu=\partial_{x_2}+y_2\partial_{z_1}+y_1\partial_{z_2}$. Then
\begin{align*}
S(R,\nu)\equiv\underbrace{[R,\nu]}_0+J[JR,\nu]\equiv\tilde{J}[\tilde{J}R,\nu]+\tilde{J}[KR,\nu]
\end{align*}
where we have used the fact that $J=\tilde{J}+K$ and also that the image of $K$ lies within $D$. Additionally, note that $[R,\nu]\equiv 0$ for our particular choice of $R,\nu$. Now recall (Eq. \ref{eq:LKuv}) that 
\begin{align*}
    [KR,\nu]\equiv\frac{-\tilde{S}(R,\nu)+J\tilde{S}(R,J\nu)}{2}\equiv\frac{-J[\tilde{J}R,\nu]-[\tilde{J}R,J\nu]}{2},
\end{align*}
where note that $-\tilde{S}(R,\nu)\equiv -J[\tilde{J}R,\nu]$ follows from Eq. (\ref{eq: Stilde}). Additionally, Eq. (\ref{eq: Stilde}) also implies $J\tilde{S}(R,J\nu)\equiv J[R,J\nu]-[\tilde{J}R,J\nu]\equiv -[\tilde{J}R,J\nu]$. The fact that $J[R,J\nu]\equiv 0$ follows from the fact that $R=\partial_{z_2}$ is a Reeb direction and $J\nu\in \SD$ does not have a Reeb component, thus yielding $[R,J\nu]\equiv 0$. Alternatively, this readily follows from Lemma \ref{lemma: Jv} as well.

Therefore, 
\begin{align*}
    S(R,\nu)\equiv\tilde{J}[\tilde{J}R,\nu]+J\Bigg(-\frac{J[\tilde{J}R,\nu]}{2}-\frac{[\tilde{J}R,J\nu]}{2}\Bigg)\equiv \tilde{J}[\tilde{J}R,\nu]+\frac{[\tilde{J}R,\nu]}{2}-\tilde{J}\frac{[\tilde{J}R,J\nu]}{2}
\end{align*}

where the reader should have in mind Remark \ref{Remark:JandJtilde} and, so, 
\begin{align*}
    S(R,\nu)\not\equiv 0 \Longleftrightarrow 2J[\tilde{J}R,\nu]+[\tilde{J}R,\nu]-\tilde{J}[\tilde{J}R,\tilde{J}\nu]\neq0\mod\SD.
\end{align*}
But, note that since $J:Q\to Q$ is an endomorphism that satisfies $J^2=-\Id$, then 
\begin{align}\label{ConditionNonVanishing}  S(R,\nu)\not\equiv 0 &\Longleftrightarrow -2\cdot[\tilde{J}R,\nu]+\tilde{J}[\tilde{J}R,\nu]+[\tilde{J}R,\tilde{J}\nu]\neq0\mod\SD\\
    &\Longleftrightarrow 2\cdot[\tilde{J}R,\nu]-\tilde{J}[\tilde{J}R,\nu]-[\tilde{J}R,\tilde{J}\nu]\neq0\mod\SD.
\end{align}
We will examine each of the three summands separately. For the ease of notation, we will write $\Delta=\sqrt{3-2x_2^2-x_2^4}$ and we will also write $f'$ in order to denote the derivative of an expression $f$ with respect to the $x_2$-coordinate; i.e. $f':=\frac{\partial}{\partial x_2}(f)$.

In order to perform the following computations, it is convenient to recall that:
\begin{itemize}
    \item[i)] Eq. (\ref{Jv1}): $\tilde{J}\partial_{z_1}=\frac{-(x_2^2+1)}{\Delta}\partial_{z_1}+\frac{2}{\Delta}\partial_{z_2}$.
    \item[ii)] Eq. (\ref{Jv2}): $\tilde{J}R=\tilde{J}\partial_{z_2}=\frac{-2}{\Delta}\cdot\partial_{z_1}+\frac{x_2^2+1}{\Delta}\cdot\partial_{z_2}$.
    \item[iii)] $\nu=\partial_{x_2}+y_2\partial_{z_1}+y_1\partial_{z_2}$.
    \item[iv)]  Lemma \ref{lemma: Jv}: $\tilde{J}\nu=-\left(\frac{x_2^2+1}{\Delta}\right)\cdot\left(\partial_{x_2}+y_2\partial_{z_1}+y_1\partial_{z_2}\right)+\left(\frac{2}{\Delta}\right)\cdot\left(\partial_{x_1}+y_1\partial_{z_1}+y_2\partial_{z_2}\right)$.
\end{itemize} 

\begin{align*}
2\cdot[\tilde{J}R,\nu]&=2\cdot\bigg[\frac{-2}{\Delta}\partial_{z_1}+\frac{x_2^2+1}{\Delta}\partial_{z_2},  \partial_{x_2}+y_2\partial_{z_1}+y_1\partial_{z_2}\bigg]=2\cdot\left(\frac{2}{\Delta}\right)'\partial_{z_1}-2\cdot\left(\frac{x_2^2+1}{\Delta}\right)'\partial_{z_2}.\\
    \tilde{J}[\tilde{J}R,\nu]&=\tilde{J}\bigg[\frac{-2}{\Delta}\partial_{z_1}+\frac{x_2^2+1}{\Delta}\partial_{z_2},  \partial_{x_2}+y_2\partial_{z_1}+y_1\partial_{z_2}\bigg]=\tilde{J}\bigg[\left(\frac{2}{\Delta}\right)'\partial_{z_1}-\left(\frac{x_2^2+1}{\Delta}\right)'\partial_{z_2}\bigg]=\\
    &=\left(\frac{2}{\Delta}\right)'\tilde{J}({\partial_{z_1}})-\left(\frac{x_2^2+1}{\Delta}\right)'\tilde{J}({\partial_{z_2}})=\left(\frac{2}{\Delta}\right)'\cdot\left(\frac{-(x_2^2+1)}{\Delta}\partial_{z_1}+\frac{2}{\Delta}\partial_{z_2}\right)\\ 
    &-\left(\frac{x_2^2+1}{\Delta}\right)'\left(\frac{-2}{\Delta}\partial_{z_1}+\frac{x_2^2+1}{\Delta}\partial_{z_2}\right)=\left(-\left(\frac{2}{\Delta}\right)'\cdot\left(\frac{x_2^2+1}{\Delta}\right)+\left(\frac{x_2^2+1}{\Delta}\right)'\cdot\left(\frac{2}{\Delta}\right)\right)\cdot\partial_{z_1}\\&+\left(\left(\frac{2}{\Delta}\right)'\cdot\left(\frac{2}{\Delta}\right)-\left(\frac{x_2^2+1}{\Delta}\right)'\cdot\left(\frac{x_2^2+1}{\Delta}\right)\right)\cdot\partial_{z_2}.\\[5pt]
    [\tilde{J}R, \tilde{J}\nu]&=\bigg[\frac{-2}{\Delta}\partial_{z_1}+\frac{x_2^2+1}{\Delta}\partial_{z_2}, -\left(\frac{x_2^2+1}{\Delta}\right)\cdot\left(\partial_{x_2}+y_2\partial_{z_1}+y_1\partial_{z_2}\right)+\left(\frac{2}{\Delta}\right)\left(\partial_{x_1}+y_1\partial_{z_1}+y_2\partial_{z_2}\right)\bigg]=\\
    &=\bigg[\frac{-2}{\Delta}\partial_{z_1}+\frac{x_2^2+1}{\Delta}\partial_{z_2}, -\frac{x_2^2+1}{\Delta}\cdot\partial_{x_2}\bigg]=-\left(\frac{x_2^2+1}{\Delta}\right)\left(\frac{2}{\Delta}\right)'\cdot\partial_{z_1}+\left(\frac{x_2^2+1}{\Delta}\right)\left(\frac{x_2^2+1}{\Delta}\right)'\cdot\partial_{z_2}.
    \end{align*}
Observe that 
\begin{align*}
    &\left(\frac{2}{\Delta}\right)'=\frac{\partial}{\partial x_2}\left(\frac{2}{\sqrt{3-2x_2^2-x_2^4}}\right)=\frac{4x_2(1+x_2^2)}{\sqrt{(3-2x_2^2-x_2^4)^3}},\\ &\left(\frac{x_2^2+1}{\Delta}\right)'=\frac{\partial}{\partial x_2}\left(\frac{1+x_2^2}{\sqrt{3-2x_2^2-x_2^4}}\right)=\frac{8x_2}{\sqrt{(3-2x_2^2-x_2^4)^3}}.
\end{align*}
In particular, the following relation holds: $\left(\frac{2}{\Delta}\right)'=\frac{x_2^2+1}{2}\left(\frac{x_2^2+1}{\Delta}\right)'$.
Therefore, if we write $2\cdot[\tilde{J}R,\nu]-\tilde{J}[\tilde{J}R,\nu]-[\tilde{J}R,\tilde{J}\nu]=A_1\cdot\partial_{z_1}+A_2\cdot\partial_{z_2}\mod\SD $, then note that its $\partial_{z_1}$-component $A_1$ corresponds to
\begin{align*}
    A_1&=2\cdot\left(\frac{2}{\Delta}\right)'-\left(-\left(\frac{2}{\Delta}\right)'\cdot\left(\frac{x_2^2+1}{\Delta}\right)+\left(\frac{x_2^2+1}{\Delta}\right)'\cdot\left(\frac{2}{\Delta}\right)\right)+\left(\frac{x_2^2+1}{\Delta}\right)\left(\frac{2}{\Delta}\right)'=\\
    &=2\cdot\left(\frac{2}{\Delta}\right)'+\left(\left(\frac{2}{\Delta}\right)'\cdot\left(\frac{x_2^2+1}{\Delta}\right)-\left(\frac{2}{\Delta}\right)'\cdot\left(\frac{2}{x_2^2+1}\right)\cdot\left(\frac{2}{\Delta}\right)\right)+\left(\frac{x_2^2+1}{\Delta}\right)\left(\frac{2}{\Delta}\right)'=\\
    &=\left(\frac{2}{\Delta}\right)'\cdot\left(2-\frac{4}{(x_2^2+1)\Delta}+2\frac{x^2_2+1}{\Delta}\right)= \frac{4x_2(1+x^2_2)}{\Delta^3}\cdot\left(2-\frac{4}{(1+x_2^2)\Delta}+\frac{2(1+x_2^2)}{\Delta}\right)\\
    &=\frac{8x_2}{\Delta^4}\cdot\left(\Delta(1+x^2_2)-2+(1+x^2_2)^2\right)=8x_2\frac{(1+x^2_2)(1+x^2_2+\Delta)-2}{\Delta^4}.
\end{align*}

If we show that $A_1$ does not identically vanish on any local neighborhood, then the main theorem will follow. We could argue analogously with $A_2$. We will actually prove that $A_1$ does not identically vanish on any arbitrarily small open neighborhood $\Op(p)\subset\R^6$ around any point $p\in\{|x_2|<1\}$. An analogous argument can be carried out for $A_2$ although, since it is not necessary for our purposes, we leave it as an exercise for the interested reader.

\begin{lemma}\label{lemma:A1NotVanishing}
The expression $A_1$, when regarded as a real function $A_1:(-1,1)\to\R$ depending on the $x_2$-variable, only vanishes at $x_2=0$.
\end{lemma}
\begin{proof}
    Write $f(x_2):=\left((1+x^2_2)(1+x^2_2+\Delta)-2\right)$ and note that the expression for $A_1$ that we obtained earlier can be expressed as $A_1=\frac{8x_2}{\Delta^4}\cdot\left((1+x^2_2)(1+x^2_2+\Delta)-2\right)=\frac{8x_2}{\Delta^4}\cdot f(x_2)$.

    On the other hand, we will show that $f(x_2)$ does not vanish. Note that $\left((1+x^2_2)(1+x^2_2+\Delta)-2\right)=0\Longleftrightarrow(1+x^2_2)\left(1+x^2_2+\sqrt{4-(1+x^2_2)^2}\right)=2$. Now, write $z=1+x_2^2$ and thus this condition reads as 
    $z(z+\sqrt{4-z^2})=2\Longleftrightarrow \sqrt{4-z^2}=\frac{2-z^2}{z}$. Taking squares, $z^2(4-z^2)=(2-z^2)^2\Longleftrightarrow -z^4+4z^2=z^4-4z^2+4\Longleftrightarrow 2(z^4-4z^2+2)=0\Longleftrightarrow z^2=2\pm\sqrt{2}\Longleftrightarrow (1+x^2_2)^2=2\pm\sqrt{2}$. None of these potential solutions are valid. On one hand, $(1+x^2_2)^2\geq 1$ whereas  $2-\sqrt{2}<1$. This discards the first solution. On the other hand, the left-hand side of the previous equality $\sqrt{4-z^2}=\frac{2-z^2}{z}$ was non-negative, which then forces $\frac{2-z^2}{z}$ to be non-negative or, equivalently, $2-z^2\geq 0$ (since $z=(1+x^2_2)>0$). Nonetheless, the solution $z^2=2+\sqrt{2}$ is incompatible with this since it would imply that $2-z^2=-\sqrt{2}$, thus discarding the second solution.
    
In order to conclude, we just observe that the expression $\frac{8x_2}{\Delta^4}$ only vanishes at the origin $0\in(-1,1)$ whereas $f(x_2)$ does not  vanish as we just checked. This yields the claim.\end{proof}

As a consequence of the previous lemma we get the following result.
\begin{proposition}\label{prop:condiii}
The tensor $S$  (Eq. \ref{eq: S}) is not identically zero on any non-empty open subset of $\{|x_2|<1\}\subset\R^6$ and, thus, $(\{|x_2|<1\},\SD)$ does not support a complex contact structure around any point, not even locally nor up to diffeomorphism.
\end{proposition}
\begin{proof}
We checked (see Eq. \ref{ConditionNonVanishing}) that for the choice $R=\partial_{z_2}$ and $\nu=\partial_{x_2}+y_2\partial_{z_1}+y_1\partial_{z_2}$, $S(R,\nu)\neq 0\mod\SD$ if and only if \begin{align*}
    S(R,\nu)\not\equiv 0 \Longleftrightarrow 2\cdot[\tilde{J}R,\nu]-\tilde{J}[\tilde{J}R,\nu]-[\tilde{J}R,\tilde{J}\nu]\neq0\mod\SD.
\end{align*}
On the other hand, we saw that if we write 
\begin{align}\label{decompositionA1A2}
    &2\cdot[\tilde{J}R,\nu]-\tilde{J}[\tilde{J}R,\nu]-[\tilde{J}R,\tilde{J}\nu]\equiv A_1\cdot\partial_{z_1}+A_2\cdot\partial_{z_2},
\end{align}
then, as a consequence of Lemma \ref{lemma:A1NotVanishing}, we conclude that the $\partial_{z_1}$-component $A_1$ of the expression ($\ref{decompositionA1A2}$), which only depends on the  $x_2$-variable, does not identically vanish on any non-empty open subset of $\{|x_2|<1\}\subset\R^6$. Therefore, neither does $S(R,\nu)$. This yields the first part of the claim which, together with Theorem \ref{thm:LocalComplexContact}, completes the proof of the second part as well.
\end{proof}

We proceed now with the proof of Theorem \ref{mainthm}.

\begin{proof}[Proof of Theorem \ref{mainthm}]
Note that condition $i)$ follows from Lemma \ref{lemma:fatness}, condition $ii)$ follows from Lemma \ref{lemma:reeb} and, finally, condition $iii)$ follows from Proposition \ref{prop:condiii}. 
\end{proof}

After developing the necessary technical results and having proved Theorem \ref{mainthm}, we close the section by stating and proving the main result in this article. We provide a global fat $(4,6)$-distribution on $\R^6$ possessing two Reeb directions and such that the germ it defines around any point $ q\in\R^6$ is not induced by a complex contact structure, not even up to local diffeomorphism.

\begin{theorem}\label{thm:global}
    Consider the global distribution $\left(\mathbb{R}^6, \xi=\ker(\beta_1)\cap\ker(\beta_2)\right)$ where
    \begin{align*}
        \beta_1&=dz_1-y_1dx_1-\frac{2y_2}{\pi(1+x_2^2)}dx_2-\left(\frac{8}{3\pi^3}\arctan^3(x_2)+\frac{2}{\pi}\arctan(x_2)+2x_1\right)dy_1,\\
        \beta_2&=dz_2-y_2dx_1-\frac{2y_1}{\pi(1+x_2^2)}dx_2.
    \end{align*}
    The associated distribution germ around any point $q\in\R^6$ satisfies the following three properties: 
    \begin{itemize}
        \item[i)] It is a corank-$2$ fat distribution germ.
        \item[ii)] It admits two Reeb directions given by $X_1=\partial_{z_1}$ and $X_2=\partial_{z_2}$.
        \item[iii)] Any diffeomorphic germ does not support a complex contact structure.
    \end{itemize}
\end{theorem}
\begin{proof}

Note that the diffeomorphism 

\begin{equation*}\label{eq:p}
		\begin{array}{rccl}
		\phi\colon & \{|x_2|<1\}\subset\R^6 & \longrightarrow &  \R^6\\
		& (x_1,x_2,y_1,y_2,z_1,z_2)& \longmapsto &(x_1, \tan(\frac{\pi}{2}x_2),y_1,y_2,z_1,z_2)
		\end{array}
		\end{equation*} 
        
        is a diffeomorphism between the open set $\{|x_2|<1\}$ and $\R^6$. On the other hand, being fat and admitting local Reeb directions are local conditions preserved under diffeomorphisms. Therefore, every germ of the distribution $\xi:=\phi_*(\SD)$  around any point $q\in\R^6$ satisfies the three properties $i), ii)$ and $iii)$, where $\SD=\ker(\lambda_1)\cap\ker(\lambda_2)$ is given by the $1$-forms in (\ref{lambda1}),      (\ref{lambda2}) below. This follows from Theorem \ref{mainthm}. Let us just show that $\xi=\phi_*(\SD)$ in order to conclude. Since $\SD$ is defined by the $1$-forms
         \begin{align}\label{lambda1}
        \lambda_1&=dz_1-y_1dx_1-y_2dx_2-\left(\frac{x_2^3}{3}+x_2+2x_1\right)dy_1\\
        \lambda_2&=dz_2-y_2dx_1-y_1dx_2,\label{lambda2}
    \end{align}
    then $\phi_*(\SD)=\ker\left((\phi^{-1})^*\lambda_1 \right)\cap\ker\left((\phi^{-1})^*\lambda_2\right)$. 
    
    Write $(\tilde{x}_1,\tilde{x}_2,\tilde{y}_1,\tilde{y}_2,\tilde{z}_1,\tilde{z}_2)=\phi(x_1,x_2,y_1,y_2,z_1,z_2)$; i.e. \[
    \tilde{x}_1=x_1,\quad \tilde{x}_2=\tan\left(\frac{\pi}{2}\cdot x_2\right),\quad \tilde{y}_i=y_i,\quad \tilde{z}_i=z_i,\quad i=1,2.
    \]
    Then we have that the inverse map is given by 
    \[x_2=\frac{2}{\pi}\arctan\left(\tilde{x}_2\right), \quad x_1=\tilde{x}_1, \quad y_i=\tilde{y}_i,\quad z_i=\tilde{z}_i,\quad i=1,2. \]

    Finally, note that $dx_2=\frac{2}{\pi}\frac{1}{1+\tilde{x}_2^2}d\tilde{x}_2,\quad dx_1=d\tilde{x}_1,\quad dy_1=d\tilde{y}_1,\quad dz_1=d\tilde{z}_1, \quad$ and then 

\begin{align*}
    \beta_1&=(\phi^{-1})^*\lambda_1=d\tilde{z}_1-\tilde{y}_1d\tilde{x}_1-\frac{2\tilde{y}_2}{\pi(1+\tilde{x}_2^2)}d\tilde{x}_2-\left(\frac{8}{3\pi^3}\arctan^3(\tilde{x}_2)+\frac{2}{\pi}\arctan(\tilde{x}_2)+2\tilde{x}_1\right)d\tilde{y}_1,\\
        \beta_2&=d\tilde{z}_2-\tilde{y}_2d\tilde{x}_1-\frac{2\tilde{y}_1}{\pi(1+\tilde{x}_2^2)}d\tilde{x}_2, \text{ thus yielding the claim.}
\end{align*}\end{proof}

\end{document}